\documentclass{amsart}

\usepackage{amsmath,amsfonts,amssymb,amsthm}
\usepackage[colorlinks]{hyperref}
\usepackage{subcaption}

\usepackage{tikz}

\newcommand\gl{\mathrel{\mathcal{L}}}
\newcommand\gr{\mathrel{\mathcal{R}}}

\newcommand\mpigr[1][S]{\ensuremath{\mathbf{mPiG}_{\gr}(#1)}}
\newcommand\mpigl[1][S]{\ensuremath{\mathbf{mPiG}_{\gl}(#1)}}
\newcommand\mpig[1][S]{\ensuremath{\mathbf{mPiG}(#1)}}
\newcommand\smpigr[1][S]{\ensuremath{\mathbf{smPiG}_{\gr}(#1)}}
\newcommand\smpigl[1][S]{\ensuremath{\mathbf{smPiG}_{\gl}(#1)}}
\newcommand\smpig[1][S]{\ensuremath{\mathbf{smPiG}(#1)}}
\newcommand\lpg[1][S]{\ensuremath{\mathbf{PiG}_{\gl}}(#1)}

\DeclareMathOperator{\dom}{\mathrm{dom}}
\DeclareMathOperator{\image}{\mathrm{image}}
\DeclareMathOperator{\rank}{\mathrm{rank}}

\theoremstyle{plain}
\newtheorem{theorem}{Theorem}[section]
\newtheorem{proposition}[theorem]{Proposition}
\newtheorem{lemma}[theorem]{Lemma}
\newtheorem{corollary}[theorem]{Corollary}

\theoremstyle{definition}
\newtheorem{definition}[theorem]{Definition} 

\theoremstyle{remark}
\newtheorem{remark}[theorem]{Remark}
\newtheorem{example}[theorem]{Example}

\title[Principal ideal graphs of Inverse Semigroups]{On the structure and  skeletals  of  principal ideal graphs of inverse semigroups}
\author{C S Preenu}
\address{University College Thiruvananthapuram, Kerala, India, 695034}
\email{cspreenu@gmail.com}
\author{R S Indu}
\address{University College Thiruvananthapuram, Kerala, India, 695034}
\email{indurs@universitycollege.ac.in}
\author{Sijo George}
\address{University College Thiruvananthapuram, Kerala, India, 695034}

\begin{document} 

\begin{abstract}
    The principal left ideal graph of a semigroup is a simple graph whose vertices are the non-zero elements of the semigroup, and two vertices are adjacent if their principal left ideals intersect non-trivially. In this paper, we study the structure of the principal ideal graphs of inverse semigroups, particularly symmetric inverse semigroups. We also introduce the concept of skeletal of a graph and show that the principal ideal graph of an inverse semigroup has a skeletal, which is a simple graph with vertex set as $\gl$ classes of non-zero elements. It is also proved that the principal ideal graph of symmetric inverse semigroups has a skeletal which is isomorphic to the intersection graph on the power set of a non-empty set.  
\end{abstract}

\maketitle

\tableofcontents

\section{Introduction}
For a semigroup $S$, the principal left ideal graph $\lpg$ of $S$, defined by the second author and L John \cite{indu2012principal}, is a simple graph with vertex set $S$  and two vertices are adjacent if and only if $S^1a\bigcap S^1b\neq \emptyset$.  Principal right ideal graphs are defined accordingly. It is evident that if $S$ has a two-sided zero,  then $0$ is an element of every principal left (similarly right) ideal of $S$. In this case, $S^1a\bigcap S^1b$ is always nonempty and hence the $\lpg$ is always complete. To overcome this challenge, we modify the definition of principal ideal graph for a semigroup $S$ with or without $0$. We define the modified principal left ideal graph $\mpigl$, with the set of all non-zero elements as vertices, and two elements $a$ and $b$ are adjacent if and only if $S^1a\bigcap S^1b$ contains a non-zero element. $\mpigr$ is defined similarly.  
 Some preliminary concepts are provided in Section 2 and are used in the sequel. In Section \ref{sec:pig}, characteristics of $\mpigl$ and $\mpigr$ are presented. The idea of skeletal and skeleton is introduced in Section \ref{sect:skeletal}, which are homomorphic images of a graph where the homomorphism preserves both adjacency and non-adjacency. A skeletal $\smpigl$ of $\mpigl$ is also discussed in Section \ref{sect:skeletal}.  
In Section \ref{sect:inversesemigroup}, $\mpig$ and $\smpig$ of inverse semigroups are discussed with special focus on symmetric inverse  and Brandt semigroups. 
\section{Preliminaries}
All the results related to semigroup theory and graph theory are based on \cite{howie1995fundamentals,clifford1961algebraic} and \cite{bondy1976graph} respectively.
\emph{Inverse semigroups} are semigroups $S$ in which for every element $x \in S$, there exists a unique  $y \in S$ such that $ x = x y x$ and $ y = y x y$, with $y$ denoted as $x^{-1}$. An element $ e \in S$ is called an idempotent if $ e^2 = e $, and the set of all idempotents $E(S)$ within an inverse semigroup forms a commutative sub-semigroup called a semilattice. For elements  $a, b \in S$, Green's relations $\gl$ and $\gr$ are defined by $a \gl b$ if and only if $S^1 a = S^1 b $, and $ a \gr b$ if and only if $ a S^1 = b S^1 $, where $ S^1$ denotes  $S$ with an identity adjoined if needed; in inverse semigroups, each $\gl$- and $\gr$-class contains exactly one idempotent, specifically $s^{-1} s$ and $s s^{-1}$, respectively, for the class of $s \in S$. The \emph{symmetric inverse} semigroup $\mathcal{I}_X$ on a set $X$ consists of all partial bijections on $X$  under composition, serving as a prototypical example of an inverse semigroup.  
Let $G$ be a group and $I$ a non-empty set.  The \emph{Brandt semigroup} \cite{volkov2019identities} $B(G, I)$ is the set 
$B(G, I) = (I \times G \times I) \cup \{0\}$
with multiplication defined by
$$(i, a, j)(k, b, n) = 
\begin{cases}
(i, a b, n) & \text{if } j = k, \\
0 & \text{otherwise},
\end{cases}$$
and  $0 \cdot x = x \cdot 0 = 0 $ for all $x \in B(G, I).$
 Also, $B(G,I)$ is a completely $0$-simple semigroup and serves as a fundamental example of inverse semigroups with zero.

\section{Principal ideal graphs}\label{sec:pig}
Let $S$ be a semigroup and $a\in S$. The principal left ideal generated by $a$ is $S^1a=\{sa:s\in S \}\bigcup \{a\}$. The \emph{principal left ideal graph}, denoted by $\mpigl$, of $S$ is a simple graph with the set of non-zero elements of $S$ as the vertex set. Two elements $a$ and $b$ are adjacent in $\mpigl$ if and only if $S^1a\bigcap S^1b$ contains a non-zero element. The principal right ideal graph, denoted by $\mpigr$, of a semigroup $S$ can be defined dually, using principal right ideals of $S$. 
\begin{theorem}
    Let $S$ be a semigroup, then we have the following:
    \begin{enumerate}
        \item If $a,b\in S$ are $\gl$ related, then they are adjacent in $\mpigl$ \label{one}. 
        \item If $a, b\in S$ are $\gr$ related, then they are adjacent in $\mpigr$\label{two}.
        \item If $S$ is a monoid, then $\mpigl$ is always connected; the converse need not be true\label{three}.
    \end{enumerate}
\end{theorem}

\begin{proof}
    Let $S$ be a semigroup and $a,b\in S$. If $a$ and $b$ two non-zero $\gl$-related elements, their principal left ideals are same. Thus, their intersection contains nonzero elements (at least $a$). Therefore, $a$ and $b$ are adjacent in $\mpigl$ and this proves (\ref{one}). Suppose $S$ contains an identity element $1$. Then $Sa = S1\bigcap Sa$ contains at least $a$. Thus, for any nonzero $a$ is adjacent to $1$ in $\mpigl$. Therefore, $\smpigl$ is connected and hence  (\ref{three}) holds. 

    Let $S$ be the left zero semigroup. That is, $ab=a$ for all $a,b\in S$. Then $\mpigl[S^0]$ is a complete graph, where $S^0 = S\bigcup \{0\}$ is  the semigroup $S$ together with 0 is attached.    
\end{proof}

Recall that an involuion $*$ on a semigroup $S$ is an anti-isomorphism on $S$. That is, for all $a,b\in S$,  
$$(ab)^* = b^*a^*\quad \text{ and }\quad (a^*)^* = a.$$
\begin{theorem}
    If $S$ is a semigroup with an involution $*$, then the graphs $\mpigl$ and $\mpigr$ are isomorphic. 
\end{theorem}
\begin{proof}
    Let $S$ be a semigroup with an involution. Then note that $a\gl b$ if and only if $a^*\gr b^*$, for any $a,b\in S$. Thus $a\mapsto a^*$ is an isomorphism from $\mpigl$ to $\mpigr$.  
\end{proof}

\section{Skeletal of a graph} \label{sect:skeletal}
In organic chemistry, a skeletal formula, also called bond-line formula or line-angle formula, is a shorthand way to represent the structure of organic molecules. It simplifies drawing complex molecules by showing only the carbon-carbon bonds as lines, and it implies the presence of hydrogen atoms without explicitly drawing them. Motivated by this, here we introduce the concept of skeletal of a graph for representing the simplified version of a graph.

\begin{definition}
    A graph $H$ is said to be a  \emph{skeletal} of a graph $G$ if there is an onto graph homomorphism $\phi:G\to H$ such that for any vertices $a$ and $b$ of $G$ are adjacent in $G$ if and only if $\phi(a)$ and $\phi(b)$ are same or adjacent in $H$. 
\end{definition}

 If $H$ is a skeletal of a graph $G$ under the homomorphism $\phi:G\to H$, then there is a subgraph $H'$ of $G$ such that the restriction of $\phi$ to $H'$ is an isomorphism between $H'$ and $H$.  Recall that an induced subgraph of a graph $G$ is a subgraph formed by a subset of the vertices of $G$ and all the edges connecting pairs of vertices in that subset. For any vertex $v\in vH$, the induced subgraph formed by $\phi^{_1}(v)$ of $G$ is a complete graph.   
\begin{example}
    Consider the graph $G$ given in Figure \ref{fig:example}. The graph $H$ is a skeletal of $G$ under the homomorphism $\phi:G\to H$ defined by $\phi(a)=\phi(b)=\phi(c) = u$,  and $\phi(d)=v$.
    \begin{figure*}[h!]
            \begin{subfigure}[b]{0.4\textwidth}
                \centering
                \begin{tikzpicture}[scale=1]
                    \node[circle,fill, draw,pink] (a) at (.5,-1) {\color{black}$a$};
                    \node[circle,fill,draw,pink] (b) at (1.5,0) {\color{black}$b$};
                    \node[circle,fill,draw,pink] (c) at (.5,1) {\color{black}$c$};
                    \node[circle,fill,draw,green] (d) at (3,0) {\color{black}$d$};
                    \draw (a) -- (b);
                    \draw (b) -- (c);
                    \draw (c) -- (a);
                    \draw (b) -- (d);
                    \draw (a) -- (d); 
                    \draw (c) -- (d);
                \end{tikzpicture}
                \caption*{Graph $G$}
            \end{subfigure} %
            \begin{subfigure}[b]{0.4\textwidth}
                \centering
                \begin{tikzpicture}[scale=1]
                    \node[circle,fill,draw,pink] (u) at (0,0) {\color{black}$u$};
                    \node[circle,fill,draw,green] (v) at (1.5,0) {\color{black}$v$};
                    \draw (u) -- (v);
                \end{tikzpicture}
                \caption*{Graph $H$}
            \end{subfigure}
        \caption{$H$ is a skeletal of $G$.}
        \label{fig:example}
    \end{figure*}
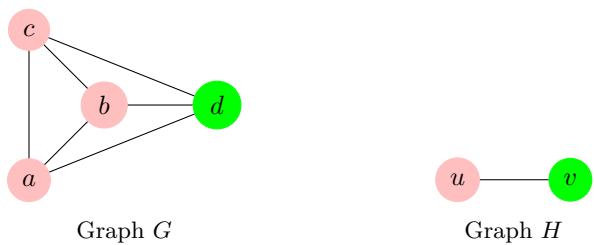
\end{example}

 A graph may have more than one skeletal. For example, the complete graph $K_n$ has $n$ different skeletals, each consisting of a single vertex. Each edge is also a skeletal of any complete graph with more than one vertex.
\begin{proposition}
    A graph $G$ with more than two vertices is complete if and only if $K_2$ is a skeletal of $G$. 
\end{proposition}
\begin{proof}
When $G$ is a complete graph with more than two vertex, then clearly an edge is a skeletal. To prove the converse, assume that $K_2$ is a skeletal of a graph $G$ under the homomorphism $\phi:G\to K_2$. Then, for any two vertices $a,b\in G$, $\phi(a)$ and $\phi(b)$ are adjacent in $K_2$. Thus $a$ and $b$ are adjacent in $G$. Hence $G$ is a complete graph.
\end{proof}

 Every graph is the skeletal of itself, under the identity isomorphism. A skeletal $H$ of $G$, under $\phi$, is said to be \emph{proper} if the inverse image of atleast one vertex of $H$ under $\phi$ contains more than one vertex of $G$. 
\begin{definition}
    A graph $G$ is said to be a \emph{skeleton} if it has no proper skeletal.
\end{definition}

\begin{theorem}
A tree with more than two vertices is a skeleton. 
\end{theorem}
\begin{proof}
Let $G$ be a tree with more than two vertices and $H$ be a proper skeletal of $G$ under the homomorphism $\phi:G\to H$. Since $G$ is not a complete graph, $H$ is not a vertex. Suppose $v\in H$ be the vertex such that $\phi^{-1}(v)$ contains more than one vertex. Since $G$ is a tree, $\phi^{-1}(v)$ contains atmost two vertices. Let $\phi^{-1}(v) =\{a,b \}$. Since $G$ is a tree, $H$ is a connected graph. Let $u$ be a vertex adjacent to $v$ in $H$. Then $\phi^{-1}(u)$, $a$ and $b$ together forms a triangle in $G$. This contradicts the fact that $G$ is a tree. Hence $G$ has no proper skeletal.  
\end{proof} 

By a similar argument, we can show that a cycle with more than three vertices is a skeleton. We state it as a proposition. 
\begin{proposition}
    Every cycle with more than three vertices is a skeleton. 
\end{proposition}

\begin{theorem}
    Let $H$ is a skeletal of a graph $G$ under the homomorphism $\phi:G\to H$, then we have the following:
    \begin{enumerate} 
        \item\label{one1} $[\phi^{-1}(v)]_G$ is a complete subgraph of $G$, for each vertex $v$ of $H$.   
        \item\label{two2} There is a subgrah $H'$, isomorphic ot $H$, of $G$.  
        \item\label{three3} If $K$ is a skeletal of $H$ under the homomorphism $\psi:H\to K$, then $K$ is a skeletal of $G$ under the homomorphism $\psi\phi:K\to G$. 
    \end{enumerate}
\end{theorem}
\begin{proof}
Let $H$ be a skeletal of a graph $G$ under the homomorphism $\phi:G\to H$.  

(\ref{one1}) Let $v$ be a vertex of $H$. If $a,b\in \phi^{-1}(v)$, then $\phi(a)=\phi(b)$. Since $\phi$ preserves adjacency, $a$ and $b$ are  adjacent in $G$. Therefore, $\phi^{-1}(v)$ induces a complete subgraph of $G$. 

(\ref{two2}) For each $v\in H$,  select a vertex $u_v\in \phi^{-1}(v)$. Let $H'$ be the subgraph of $G$ induced by $\{u_v:v\in H \}$. Then, $H'$ and $H$ are isomorphic, since $\phi$ preserves both adjacency and non-adjacency. 

(\ref{three3}) Let $\psi:H\to K$ be a homomorphism under which $K$ is a skeletal of $H$. Then for any two vertices $a,b\in G$, $\psi\phi(a)$ and $\psi\phi(b)$ are adjacent in $K$ if and only if $\phi(a)$ and $\phi(b)$ are adjacent in $H$, if and only if $a$ and $b$ are adjacent in $G$. Similarly, we can show that $\psi\phi(a)$ and $\psi\phi(b)$ are not adjacent in $K$ if and only if  $\phi(a)$ and $\phi(b)$ are not adjacent in $H$, if and only if  $a$ and $b$ are not adjacent in $G$. Thus $\psi\phi:G\to K$ is a homomorphism which preserves both adjacency and non-adjacency. Hence, the result follows.
\end{proof}

The next theorem connects the existence of a proper skeletal with the eigenvalues of the various matrices associated with the graph. For a graph $G$, the adjacency matrix, denoted by $A$, is the  0-1 matrix indexed by the vertex set of $G$, where the entry $A_{ij}$ is 1 if the vertices $i$ and $j$ are adjacent, and 0 otherwise. The Laplacian matrix $L$ of a graph $G$ is defined as $L = D - A$, where $D$ is the degree matrix, a diagonal matrix where each diagonal entry $D_{ii}$ represents the degree of vertex $i$. The signless Laplacian matrix $Q$ is defined as $Q = D + A$. The eigenvalues of these matrices provide insights into various properties of the graph, including its connectivity and structure \cite{das2004laplacian,merris1994laplacian,cvetkovic2010towards}.

\begin{theorem}
    Let  $G$ be a graph and $H$ be a skeletal under the homomorphism $\phi:G\to H$. Suppose  $v\in V(H)$ with  $k=|\phi^{-1}(v)|\geq 2$ and $s=deg(v)$ be the degree of $v$ in $H$. Then the following hold:
    \begin{enumerate}
        \item $-1$ is an eigenvalue of $A$.
        \item $s+1$ is an eigenvalue of $L$ with multiplicity atleast $k-1$.
        \item $s-1$ is an eigenvalue of $Q$ with multiplicity atleast $k-1$. 
    \end{enumerate} 
\end{theorem}
\begin{proof}
Let $H$ be a proper skeletal of the graph $G$ under the homomorphism $\phi:G\to H$ and $v$ be a vertex of $H$ such that $\phi^{-1}(v)$ contains at least two vertices of $G$.

Consider the adjacency matrix $A$ of $G$. The rows corresponding to any two vertices of $\phi^{-1}(v)$ are identical, except the diagonal entry. Specifically, these rows have zeros on the diagonal and identical off-diagonal entries 1 due to the identical adjacency relations. Adding the identity matrix $I$ shifts the diagonal entries from $0$ to $1$, so the corresponding rows in $A + I$ are identical. Hence, $A + I$ has at least $k$ identical rows, which are linearly dependent, implying that $A + I$ is singular. Consequently, $-1$ is an eigenvalue of $A$ with multiplicity at least $k-1$.

In the case of the Laplacian matrix, the rows corresponding to any two vertices of $\phi^{-1}(v)$ are identical, except the diagonal entry, which contains the degree $s = \deg(v)$, and have $-1$ entries in positions corresponding to adjacent vertices. Thus, $L-(s+1)I$ has at least two identical rows and hence singular. Therefore, $deg(v)+1$ is an eigenvalue of $L$. Also, if $|\phi^{-1}(v)|=k$, there are $k$ identical rows in $L-(s+1)I$. Thus, the multiplicity of the eigenvalue $deg(v)+1$ is atleast $k-1$. In a similar manner, we can prove the last case also. 
\end{proof}

\subsection{A Skeletal of $\mpig$}
Let $S$ be a semigroup. For $a\in S$, let $L_a$ denotes the $\gl$ class containing $a$. Consider the simple graph $\smpigl$ with vertex set $\{ L_a:0\ne a\in S\}$. Two vertices $L_a$ and $L_b$ are adjacent in $\smpigl$ if and only if $a$ and $b$ are adjacent in $\mpigl$. Similarly, we can also define $\smpigr$. 
\begin{proposition}
    For any semigroup $S$, $\smpigl$ is a skeleton of $\mpigl$.
\end{proposition}
\begin{proof}
    By the definition of $\smpig$ and $\mpig$, it is clear that $a\mapsto L_a$ is a graph homomorphism, which makes $\smpig$ a skeletal of $\mpig$. 
\end{proof}

\section{$\mpigl$ and $\smpigl$ of Inverse Semigroups}\label{sect:inversesemigroup} 

In this section $S$ denotes an invese semigroup. The map $a\mapsto a^{-1}$ is an involution on $S$ and hence $\mpigl$ and $\mpigr$ are isomorphic. Also, we have $Se\bigcap Sf=Sef$ for any idempotents $e,f\in S$ \cite{howie1995fundamentals}. The follwoing lemma is a consequence of this fact. 
\begin{lemma}
    Two non-zero idempotents $e$ and $f$ are adjacent in $\mpig$ if and only if $ef\ne 0$. 
\end{lemma}
\begin{proof}
Suppose $e$ and $f$ are two non-zero idempotents in an inverse semigroup $S$. We have, $Se\bigcap Sf=Sef$. Since $ef\in Sef$. we can see that $Sef=\{0\}$ if and only if $ef=0$. Thus $e$ and $f$ are adjacent if and only if $ef\ne 0$. 
\end{proof}

Now we prove the condition for the adjacency in general. 
\begin{theorem}
    Let $S$ be an inverse semigroup and two non-zero elements $x$ and $y\in S$ are adjacent in $\mpig$ if and only if 
    $$xy^{-1}\ne 0.$$ 
\end{theorem}
\begin{proof}
    Suppose $x$ and $y$ are two non-zero elements in an inverse semigroup $S$. Then there exist unique idempotents $e=x^{-1}x$ and $f=y^{-1}y$ such that $Sx=Se$ and $Sy=Sf$. Thus we have 
    $$Sx\cap Sy=Se\cap Sf=Sef.$$
    But we have $Sef\ne \{0\}$ if and only if $ef\ne 0$. Therefore $x$
    and $y$ are adjacent in $\mpig$ if and only if $ef\ne 0$, that is if and only if $x^{-1}xy^{-1}y\ne 0.$ But $x^{-1}xy^{-1}y= 0$ if and only if $xy^{-1}=0$ (multiplying with $x$ to the left and $y^{-1}$ to the right). Thus we have $x$ and $y$ are adjacent in $\mpig(S)$ if and only if $xy^{-1}\ne 0$. 
\end{proof}
Since each $\gl$ class contains unique idempotents, we can identify the vertices of $\smpig$ with non-zero idempotents of $S$. Let  $e$ and $f$ are two non-zero idempotents in  $S$, then $L_e$ and $L_f$ are adjacent in $\smpig$ if and only if $ef\ne 0$ in $S$.    Now we state the discussion as a theorem.
\begin{theorem}
    Let $S$ be an inverse semigroup with the semilattice of idempotents $E$. Then
    \begin{enumerate}
        \item $\{L_e: e\in E\}$ is the vertex set of $\smpigl$.
        \item $L_e$ and $L_f$ are adjacent in $\smpigl$ if and only if $ef\ne 0$. \qed 
    \end{enumerate}    
\end{theorem}

From the above discussions, we can say that if $S$ is an inverse semigroup without a zero, then $\mpigl$, $\mpigr$, $\smpigl$ and $\smpigr$ are complete graphs.  The Zero-divisor graph of a ring $R$ is defined as a simple graph with the set of all non-zero zero-divisors as the vertex set and two vertices are adjacent if and only if their product is zero \cite{anderson99zerodivisor}.
\begin{remark}\label{remark-one}
  We can identify the skeletal $\smpigl$ of an inverse semigroup $S$ as follows: the semilattice $E=E(S)$ is the vertex set and two vertices are adjacent if and only if their product is non-zero. Thus, the complement of $\smpigl$ is a typical zero-divisor graph, with the only difference being that its vertex set contains all nonzero elements, rather than just the nonzero zero-divisors. 
\end{remark}

\subsection{Symmetric inverse semigroups}
The interest in studying an abstract structure is the existence of interesting examples. One of the major examples of inverse semigroups is the symmetric inverse semigroup \cite{lipscomb1996symmetric}, the collection of all partial one-to-one mappings on a set $X$, denoted by $IS_X$. If $|X|$ is finite, with $n$ elements, we use $IS_n$ to denote $IS_X$. This section studies the properties of $\mpig(IS_n)$. To begin with, we state some properties of $IS_n$; the proof of these statements is omitted since they can be found either in  \cite{ganyushkin2008classical,reilly1965embedding} or a usual easy verification.   
\begin{proposition}[\cite{ganyushkin2008classical}]
    Let $n\in \mathbb N$. Then $IS_n$ has the following properties:
    \begin{enumerate}
        \item $|IS_n| = (n+1)^n$;
        \item $\alpha\gl \beta$ if and only if $\dom \alpha = \dom \beta$;
        \item $\alpha\in IS_n$ is an idempotent if and only if $\alpha$ is the identity on its domain. 
        \item number of idempotents in $IS_n$ is  $2^n$
    \end{enumerate}
\end{proposition}
Two elements $x$ and $y$ in an inverse semigroup $S$ are adjacent in $\mpig$ if and only if $xy^{-1}\ne 0$ in $S$. Thus in the case of $IS_n$, $xy^{-1}=0$ if and only if $\image(x) \bigcap \dom (y^{-1}) = \emptyset$.  But $\dom(y^{-1}) = \image(y)$. Thus $x$ and $y$ are adjacent if and only if $$\image(x) \bigcap \image(y)\ne \emptyset.$$  
We have proved the following result. 
\begin{theorem}
    Two vertices $x,y\in IS_n$ are adjacent in $\mpig[IS_n]$ if and only if 
    $$\image(x) \bigcap \image(y) \ne \emptyset.  $$ \qed 
\end{theorem}

Now we find the degree of each vertex in $\mpig[IS_n]$. Let us use $\rank (\alpha)$ to denote the cardinality of the image of the partial one-to-one mapping $\alpha$ on $X$. 
\begin{theorem}
    Let $\alpha\in IS_n$ be a non-zero idempotent with rank $k$. Then the degree of $\alpha$ in $\smpigl[IS_n]$ is given by 
    $$\deg(\alpha) = 2^n-2^{n-k}-1.$$
\end{theorem}
\begin{proof}
    Let $A$ be the image of the idempotent $\alpha \in IS_n$, so $|A| = k$. The non-zero idempotents in $IS_n$ correspond to the non-empty subsets of $X$, totalling $2^n - 1$. The number of idempotents whose images are disjoint from $A$ is $2^{n-k}$ (all subsets of $X \setminus A$). Therefore, the number of idempotents whose images intersect $A$ is $2^n - 2^{n-k} - 1$ (subtracting the $2^{n-k}$ disjoint subsets). Thus, the degree of $\alpha$ in $\smpigl[IS_n]$ is $2^n - 2^{n-k} - 1$. 
\end{proof}

\begin{corollary}
   The number of edges of $\smpigl[IS_n]$ is given by
   $$\frac{1}{2}\left((2^n-1)^2-(3^n-2^n)\right) $$
\end{corollary} 

\begin{proof}
    The number of edges in $\smpigl[IS_n]$ is half the sum of the degrees of all vertices. The degree of each non-zero idempotent is $2^n - 2^{n-k} - 1$, where $k$ is the rank of the idempotent. The number of non-zero idempotents with rank $k$ is $\binom{n}{k}$. Thus, the total number of edges is given by
    \begin{align*}
        \dfrac{1}{2}\sum_{k=1}^n\binom{n}{k} (2^n-2^{n-k}-1) & = \dfrac{1}{2}(2^n-1)\sum_{k=1}^n \binom{n}{k} -\dfrac{1}{2} \sum_{n=1}^k\binom{n}{k}1^k2^{n-k}\\
        & = \dfrac{1}{2}\left((2^n-1)^2-(3^n-2^n)\right). 
    \end{align*}
    Simplifying this expression gives us the final result.
\end{proof}

\begin{theorem}
    $\smpigl[IS_n]$ is isomorphic to the intersection graph of the non-empty subsets of an $n$-element set.
\end{theorem}
\begin{proof}
    Let $X$ be an $n$-element set. The intersection graph of the non-empty subsets of $X$ has vertices corresponding to each non-empty subset of $X$, with two vertices adjacent if and only if their corresponding subsets intersect. In $\smpigl[IS_n]$, the vertices correspond to non-zero idempotents of $IS_n$, which are in one-to-one correspondence with non-empty subsets of $X$ (the image of the idempotent). Two idempotents are adjacent in $\smpigl[IS_n]$ if and only if their images intersect. Thus, the adjacency condition in both graphs is the same, establishing an isomorphism between $\smpigl[IS_n]$ and the intersection graph of the non-empty subsets of an $n$-element set.
\end{proof}  
\subsection{Two extremes: Semilattices and Brandt Semigroups} 
In this section, we examine two extremal cases for principal ideal graphs of inverse semigroups: semilattices and Brandt semigroups. For semilattices without zero, the skeletal $\smpigl$ is a complete graph, since $ef\in Se\bigcap Sf=Sef$. In contrast, for Brandt semigroups, the skeletal $\smpigl$ is a null graph. By Remark \ref{remark-one}, we have the following proposition.
\begin{proposition}
    When $S$ is a semilattice, then we have the following:
    \begin{enumerate}
        \item $\mpigl$ and $\mpigr$ are isomorphic. 
        \item $\mpigl$ and $\smpigl$ are isomorphic. 
        \item If $S$ has no zero, then $\mpigl$ is a complete graph. 
    \end{enumerate}
\end{proposition}
\begin{theorem}
Let $S= B(G, I) = (I \times G \times I) \cup \{0\}$ be a Brandt semigroup.  Then $\mpigl$ decomposes as a disjoint union of complete graphs indexed by $I$.  Specifically, two vertices $a = (i, g, j)$ and $b = (k, h, l)$ are adjacent if and only if $j = l$. 
\end{theorem}

\begin{proof}
Consider two nonzero elements $a = (i, g, j)$ and $b = (k, h, l)$ in $B(G, I)$. Their principal left ideals are
\[
S^{1} a = \{ (m, x, j) : m \in I, x \in G \} \cup \{0\}, \quad
S^{1} b = \{ (n, y, l) : n \in I, y \in G \} \cup \{0\}.
\]

The intersection $S^{1} a \cap S^{1} b$ contains a nonzero element if and only if there exist $m, n \in I$ and $x, y \in G$ such that $(m, x, j) = (n, y, l)$, which implies $j = l$. Hence, $a$ and $b$ are adjacent in $\mpigl$ if and only if their right indices coincide.
This means the graph breaks into disjoint components, each consisting of all elements sharing the same right index. Each component is a complete graph because, for any two elements with the same right index, their principal left ideals are same, ensuring adjacency. Thus, $\mpigl$ decomposes into cliques indexed by $I$, with each clique corresponding to a fixed right index.
\end{proof}

Now we consider the skeletal $\smpigl$ of a Brandt semigroup $S=B(G,I)$. By Remark \ref{remark-one}, we can identify $\smpigl$ as follows:
The set of non-zero idempotents is the vertex set, and two idempotents are adjacent if and only if their product is zero. 
\begin{proposition}
    Let $S$ be a Brandt semigroup, then $\smpigl$ and $\smpigr$ are null graphs. 
\end{proposition}
\begin{proof}
    Since two distinct idempotents $(i,e,i)$ and $(j,e,j)$ are adjacent in $\smpigl$ if and only if their product is non-zero. But, in Brandt semigroup product of distinct idempotents is zero. Thus, no two distinct idempotents are adjacent and hence the result. 
\end{proof}
\nocite{*}
 \bibliographystyle{amsplain} 
\bibliography{books.bib} 
\end{document}